\documentclass[12pt, a4paper]{article}
\usepackage{etex}

\usepackage{amssymb, amscd, latexsym}
\usepackage{amsmath}
\usepackage{amsthm}
\usepackage[dvips]{graphicx}
\usepackage[all]{xy}
\usepackage{mathrsfs}
\usepackage{pstricks-add}
\usepackage[utf8]{inputenc}
\usepackage{bigstrut}
\usepackage{multirow}
\usepackage{array}
\usepackage{ctable}
\usepackage{booktabs}
\usepackage{tabularx}
\usepackage{fancyhdr}
\usepackage{float}
\usepackage{tikz}
\usepackage{hyperref}

\hypersetup{
pdftitle={PhD Thesis},
pdfauthor={Ali Akbar Yazdan Pour},
pdfsubject={Resolution and Castelnuovo-Mumford Regularity},
pdfcreator={Ali Akbar Yazdan Pour},
colorlinks,
linkcolor=blue,
citecolor=magenta,
anchorcolor=red,
bookmarksopen,
urlcolor=red,
filecolor=red,
pdfpagetransition={Wipe}
}

\newtheorem{thm}{Theorem}[section]
\newtheorem{cor}[thm]{Corollary}
\newtheorem{lem}[thm]{Lemma}
\newtheorem{prop}[thm]{Proposition}
\newtheorem{defn}[thm]{Definition}
\newtheorem{rem}[thm]{Remark}

\def\cocoa{{\hbox{\rm C\kern-.13em o\kern-.07em C\kern-.13em o\kern-.15em A}}}

\def\sqr#1#2{{\vcenter{\hrule height.#2pt
        \hbox{\vrule width.#2pt height#1pt \kern#1pt
            \vrule width.#2pt}
        \hrule height.#2pt}}}

\def\depth{{\rm depth}\,}

\def\reg{{\rm reg}\,}

\def\NN{{\mathbb N}}

\def\ZZ{{\mathbb Z}}

\newcommand{\email}[1]{\begin{flushleft}E-mail: {\ttfamily #1}\end{flushleft}}

\begin{document}

\title{\bf Regularity of edge ideal of a graph}
\author{Marcel Morales$^1$, Ali Akbar Yazdan Pour$^{1,2}$, Rashid Zaare-Nahandi$^2$\\
\small $^{1}$ Universit\'e de Grenoble I, Institut Fourier, Laboratoire de Math\'ematiques, France\\
\small $^{2}$ Institute for Advanced Studies in Basic Sciences, P. O. Box 45195-1159, Zanjan, Iran}

\date{}
\maketitle

\begin{abstract}
In this paper, we introduce some reduction processes on graphs which preserve the regularity of related edge ideals. As a consequence, an alternative proof for the theorem of R. Fr\"oberg on linearity of resolution of edge ideal of graphs is given.
\end{abstract}

\section{Introduction and Preliminaries}
Throughout this paper, we assume that $G$ is a simple finite graph on vertex set $[n] = \{1, \ldots, n\}$. A graph $G$ is called \textit{chordal}, if every induced cyclic subgraph of $G$ has length $3$. A vertex $v$ of a graph $G$ is \textit{simplicial}, if the neighborhood of $v$ in $G$ is a complete subgraph. Let $S = K[x_1, \ldots, x_n]$ be the polynomial ring over a field $K$ with standard grading. The \textit{edge ideal} of $G$ is defined by
$$I(G) = (x_ix_j \colon  \quad \{i, j\} \text{ is an edge in } G) \subset S.$$
Let $I \neq 0$ be a homogeneous ideal of $S$ and $\NN$ be the set of non-negative integers. For every $i\in \NN$, one defines:
$$t_i^S(I) = \max\{j \colon \quad \beta_{i,j}^S(I) \neq 0\}$$
where $\beta_{i,j}^S(I)$ is the $i,j$-th graded Betti number of $I$ as an $S$-module. The \textit{Castelnuovo–Mumford regularity} of $I$ is given by:
$$\reg (I) = \sup\{t^S_i (I) - i \colon \quad i \in \ZZ\}.$$
We say that the ideal $I$ has a \textit{$d$-linear resolution}, if $I$ is generated by homogeneous polynomials of degree $d$ and $\beta_{i,j}^S(I) =0$, for all $j \neq i + d$ and $i \geq  0$. For an ideal which has a $d$-linear resolution, the Castelnuovo–Mumford regularity would be $d$.

Recently, several mathematicians have studied the regularity of edge ideals of graphs. Kummini in \cite{Kummini} has computed the Castelnuovo–Mumford regularity of Cohen–Macaulay bipartite graphs and Van Tuyl in \cite{VanTuyl} has generalized it for sequentially Cohen–Macaulay bipartite graphs. In \cite{Yassemi} the regularity was computed for very well-covered graphs, in \cite{Moradi}, some bounds were obtained for the regularity of edge ideals of vertex decomposable and shellable graphs and in \cite{Woodroofe}, the Castelnuovo–Mumford regularity was calculated for edge ideals of several other classes of graphs. Also \cite{Nevo} has studied the topology of the lcm-lattice of edge ideals and derived upper bounds on the Castelnuovo–Mumford regularity of the ideals.

The Alexander dual of a square-free monomial ideals, plays an essential role in combinatorics and commutative algebra. For a square-free monomial ideal $I= \left( M_1, \ldots, M_q \right) \subset S=K[x_1, \ldots, x_n]$, the \textit{Alexander dual} of $I$, denoted by  $I^\vee$, is defined to be:
$$I^\vee= P_{M_1} \cap \cdots \cap P_{M_q}$$
where, $P_{M_i}$ is prime ideal generated by $\{x_j \colon \quad x_j | M_i \}$.

We begin with a well-known result of Eagon and Reiner and its generalization by Terai concerning the relation of the regularity of a square-free monomial ideal and the Cohen-Macaulayness of its Alexander dual. For a complete discussion of this fact, one can refer to \cite{Terai}.

\begin{thm}[{Eagon-Reiner theorem \cite[Theorem 3]{Eagon-Reiner}}] \label{Eagon-Reiner}
Let $I$ be a square-free monomial ideal in $S=K[x_1, \ldots, x_n]$. The ideal $I$ has a $q$-linear resolution if and only if $S/I^\vee$ is Cohen-Macaulay of dimension $n - q$.
\end{thm}

\begin{thm}[{\cite[Theorem 2.1]{Terai}}]\label{Terai}
Let $I$ be a square-free monomial ideal in $S=K[x_1, \ldots, x_n]$ with $\dim S/I \leq n-2$. Then,
$$\dim \frac{S}{I^\vee} - \depth \frac{S}{I^\vee}= \reg(I) - {\rm indeg}\,(I).$$
Here, ${\rm indeg}\,(I)$ indicates the initial degree of $I$. That is, the minimal degree of a minimal generator of $I$.
\end{thm}

The following lemma was proved in \cite{maar}.

\begin{lem} \label{Marcel}
Let $I, I_1$ and $T$ be ideals in a commutative Noetherian local ring $(R, \mathfrak{m})$ such that, $I=I_1+T$ and
$$r:= \depth \frac{R}{I_1 \cap T} \leq \depth \frac{R}{T}.$$
Then, for all $i<r-1$ one has:
\begin{equation*}
H_\mathfrak{m}^i \left( \frac{R}{I_1} \right) \cong H_\mathfrak{m}^i \left( \frac{R}{I} \right).
\end{equation*}
\end{lem}

\begin{rem} \rm \label{Passing to regularity}
Let $I, J$ be square-free monomial ideals generated by elements of degree $d\geq 2$ in $S=K[x_1, \ldots, x_n]$. By Theorem~\ref{Terai}, we have
\begin{equation*}
\begin{split}
\reg (I) = n - \depth\frac{S}{I^\vee}, \qquad \reg (J) = n - \depth\frac{S}{J^\vee}.
\end{split}
\end{equation*}
Therefore, $\reg (I) = \reg (J)$ if and only if $\depth {S}/{I^\vee} = \depth {S}/{J^\vee}$.
\end{rem}

For a graph $G$, let $\bar{G}$ denotes the complement of graph $G$. That is, $V(\bar{G}) = V(G)$ and
$$E(\bar{G}) = \big\{ \{i,j\} \colon \quad \{i, j\} \notin E(G) \big\}.$$
Frequently in this paper, we take a graph $G$ and we let $I= I(\bar{G})$ be the edge ideal of graph $\bar{G}$. The following proposition was proved in \cite[Proposition 4.1.1]{Jacques}.
\begin{prop}
If $H$ is an induced subgraph of $G$ on a subset of the vertices of $G$, then:
$$\beta^S_{i,j} \left( I(\bar{H}) \right) \leq \beta^S_{i,j} \left( I(\bar{G}) \right)$$
for all $i,j$.
\end{prop}

\begin{cor} \label{Induced subgraph}
Let $G$ be a graph and $H$ an induced subgraph of $G$. If $I(\bar{H})$ does not have linear resolution, then the ideal $I(\bar{G})$ does not have linear resolution.
\end{cor}

\section{Reduction processes on graphs}

In this section we introduce some reduction processes on vertices and edges of a graph which preserve the regularity of the edge ideal of the complement of the graph.

In the following, for convenience we use this notation:
\begin{equation*}
\text{\rm \textbf{x}} = {x_1, \ldots, x_n }, \quad \text{\rm \textbf{z}} = {z_1, \ldots, z_r}, \quad \text{\rm \textbf{y}} = {y_1, \ldots, y_m}.
\end{equation*}
Also for a subset $F \subset [n]$, we set $\text{\rm \textbf{x}}_F = \prod\limits_{i\in F} x_i$ and $P_F = (x_i \colon \quad i \in F )$.

\begin{lem} \label{Depth S/x_nI = Depth S/I}
Let $S=K[\text{\rm \textbf{x}}, \text{\rm \textbf{y}}]$ be the polynomial ring and $I$ be an ideal in $K[\text{\rm \textbf{y}}]$. Then,
$$\depth \frac{S}{\left( x_1\cdots x_n \right) IS} = \depth \frac{S}{IS}.$$
\end{lem}

\begin{lem} \label{Lemma for Regularity of union of graphs}
Let $I \neq 0$ be square-free monomial ideal in $K[\text{\rm \textbf{x}}, \text{\rm \textbf{z}}]$ and $J$ be the ideal
\begin{equation*}
J=I+(x_iy_j \colon 1 \leq i \leq n, 1 \leq j \leq m) \subset S:= K[\text{\rm \textbf{x}}, \text{\rm \textbf{y}}, \text{\rm \textbf{z}}].
\end{equation*}
Then, we have the followings:
\begin{itemize}
\item[\rm (i)] $J^\vee = I^\vee \cap \left( \text{\rm \textbf{x}}_{[n]}, \text{\rm \textbf{y}}_{[m]} \right)$.
\item[\rm (ii)] If $z_iz_j \notin I$ for all $1 \leq i < j \leq r$, then $\reg (I) = \reg (J)$.
\end{itemize}
\end{lem}

\begin{proof}
(i) This is an easy computation.

(ii) By Remark~\ref{Passing to regularity}, it is enough to show that, $\depth {S}/{I^\vee} = \depth {S}/{J^\vee}$. We know that $I^\vee$ is intersection of prime ideals $P_F$, such that:
$$|F|=2, \quad G(P_F) \subset \{ \text{\rm \textbf{x}}, \text{\rm \textbf{z}} \} \quad \text{and } \quad \text{\rm \textbf{x}}_{F} \in I.$$
Since $z_iz_j \notin I$, for all $1 \leq i < j \leq r$, it follows that $P \nsubseteq \{\text{\rm \textbf{z}}\}$, for all $P \in {\rm Ass}\,(I)$. Hence $\text{\rm \textbf{x}}_{[n]} \in P$, for all $P \in {\rm Ass}\,(I)$. This means that, $\text{\rm \textbf{x}}_{[n]} \in I^\vee$. Now, by part (i) of this theorem, we have:
\begin{equation}
\begin{split}
J^\vee & = I^\vee \cap \left( \text{\rm \textbf{x}}_{[n]}, \text{\rm \textbf{y}}_{[m]} \right) \\
& = \left( \text{\rm \textbf{x}}_{[n]} \right) + \left( \left( \text{\rm \textbf{y}}_{[m]} \right)I^\vee \right).
\end{split}
\end{equation}
Clearly, $\left( \text{\rm \textbf{x}}_{[n]} \right) \cap \left( \left( \text{\rm \textbf{y}}_{[m]} \right)I^\vee  \right) = \left( \text{\rm \textbf{x}}_{[n]} \text{\rm \textbf{y}}_{[m]} \right)$. Hence by Lemma~\ref{Marcel}, we have:
\begin{equation}\label{Local EQ 1.1}
H_\mathfrak{m}^i \left( \frac{S}{J^\vee} \right) \cong H_\mathfrak{m}^i \left( \frac{S}{ \left( \text{\rm \textbf{y}}_{[m]} \right)I^\vee } \right), \qquad \text{for all } i<(m+n+r) - 2.
\end{equation}
Since,
$$\dim \frac{S}{J^\vee} =(m+n+r) - 2 = \dim \frac{S}{I^\vee},$$
from (\ref{Local EQ 1.1}) and Lemma~\ref{Lemma for Regularity of union of graphs} we conclude that, $\depth S/I^\vee = \depth S/J^\vee$.
\end{proof}

\begin{thm} \label{Theorem Regularity of Union of graphs}
Let $G_1$ and $G_2$ be graphs on two vertex sets $V_1$ and $V_2$ respectively, such that $V_1 \cap V_2 =  \{\text{\rm \textbf{z}}\} $ and $\{z_i, z_j \} \in E(G_1) \cap E(G_2)$, for all $1 \leq i < j \leq r$. Let
\begin{align*}
& I_1=I(\bar{G}_1) \subset K[\text{\rm \textbf{x}}, \text{\rm \textbf{z}}], \\
& I_2=I(\bar{G}_2) \subset K[\text{\rm \textbf{y}}, \text{\rm \textbf{z}}], \\
& I=I \left( \overline{G_1 \cup G_2} \right) \subset S=K[\text{\rm \textbf{x}}, \text{\rm \textbf{y}}, \text{\rm \textbf{z}}].
\end{align*}
be corresponding non-zero circuit ideals. Then,
\begin{itemize}
\item[\rm (i)] $\depth \frac{S}{I^\vee} = \min \{ \depth \frac{S}{I_1^\vee}, \; \depth \frac{S}{I_2^\vee} \}$.
\item[\rm (ii)] $\reg (I) = \max \{ \reg (I_1), \; \reg (I_2) \}$.
\item[\rm (iii)] The ideal $I$ has a $2$-linear resolution if and only if both of $I_1$ and $I_2$ have a $2$-linear resolution.
\end{itemize}
\end{thm}

\begin{center}
\psset{xunit=1.2cm,yunit=1.2cm,algebraic=true,dotstyle=o,dotsize=3pt 0,linewidth=0.8pt,arrowsize=3pt 2,arrowinset=0.25}
\begin{pspicture*}(1.84,1.7)(6.28,4.6)
\rput{-89.32}(4.01,3.34){\psellipse[fillcolor=lightgray,fillstyle=solid,opacity=0.1](0,0)(0.87,0.24)}
\parametricplot{0.8002750058049358}{5.506718700491271}{1*1.18*cos(t)+0*1.18*sin(t)+3.17|0*1.18*cos(t)+1*1.18*sin(t)+3.33}
\parametricplot{-2.326191048163512}{2.3499994472801324}{1*1.17*cos(t)+0*1.17*sin(t)+4.82|0*1.17*cos(t)+1*1.17*sin(t)+3.35}
\rput[lt](4.86,4.04){\parbox{1.17 cm}{$y_1\\ \vdots \\ y_m$}}
\rput[lt](2.88,4.04){\parbox{1.17 cm}{$x_1\\ \vdots \\ x_n$}}
\rput[tl](3.02,2.0){$G_1$}
\rput[tl](4.78,2.0){$G_2$}
\rput[lt](3.87,3.95){\parbox{1.17 cm}{$z_1\\ \vdots \\ z_r$}}
\end{pspicture*}
\end{center}

\begin{proof}
(i) We know that:
\begin{equation*}
\begin{split}
I = I_1 + I_2 + \left(x_iy_j \colon \quad 1 \leq i \leq n, \quad 1 \leq j \leq m \right).
\end{split}
\end{equation*}
\noindent Let,
\begin{equation*}
\begin{split}
J_1 = I_1  + \left(x_iy_j \colon \quad 1 \leq i \leq n, \quad 1 \leq j \leq m \right), \\
J_2 = I_2  + \left(x_iy_j \colon \quad 1 \leq i \leq n, \quad 1 \leq j \leq m \right).
\end{split}
\end{equation*}
Then, $I^\vee = J_1^\vee \cap J_2^\vee$ and by Lemma~\ref{Lemma for Regularity of union of graphs}(ii), we have:
\begin{equation*}
J_1^\vee + J_2^\vee = \left( \text{\rm \textbf{x}}_{[n]}, \text{\rm \textbf{y}}_{[m]} \right).
\end{equation*}
From Mayer-Vietoris long exact sequence (\cite[Proposition 5.1.8.]{HHBook}), we have the long exact sequence:
\begin{align*} \label{Local Sequence 1.1}
\cdots  \rightarrow H_\mathfrak{m}^{i-1} \left( \frac{S}{ \left( \text{\rm \textbf{x}}_{[n]}, \text{\rm \textbf{y}}_{[m]} \right)} \right) \rightarrow H_\mathfrak{m}^{i} \left( \frac{S}{I^\vee} \right) \rightarrow  H_\mathfrak{m}^i \left( \frac{S}{J_1^\vee} \right) \oplus H_\mathfrak{m}^i \left( \frac{S}{J_2^\vee} \right) \rightarrow \\
\rightarrow H_\mathfrak{m}^{i} \left( \frac{S}{ \left( \text{\rm \textbf{x}}_{[n]}, \text{\rm \textbf{y}}_{[m]} \right)} \right) \rightarrow \cdots.
\end{align*}
Hence, for all $i < (m+n+r) -2$, we have:
\begin{equation*}
H_\mathfrak{m}^{i} \left( \frac{S}{J_1^\vee} \right) \oplus H_\mathfrak{m}^{i} \left(\frac{S}{J_2^\vee} \right) \cong H_\mathfrak{m}^{i} \left( \frac{S}{I^\vee} \right).
\end{equation*}
This implies that,
\begin{equation}\label{Local EQ 1.4}
\depth \frac{S}{I^\vee} = \min \{ \depth \frac{S}{J_1^\vee}, \; \depth \frac{S}{J_2^\vee} \}.
\end{equation}
By Lemma~\ref{Lemma for Regularity of union of graphs}(ii) and Remark~\ref{Passing to regularity}, we have:
\begin{equation*}
\depth \frac{S}{I_i^\vee} = \depth \frac{S}{J_i^\vee}, \qquad \text{for }i=1,2.
\end{equation*}
Hence, (i) follows from (\ref{Local EQ 1.4}) and the above equality.

(ii) This is an easy consequence of (i) and Remark~\ref{Passing to regularity}.

(iii) This is a direct consequence of (ii).
\end{proof}

\begin{lem}\label{lemma for edge subdivision}
Let $G$ be a graph on vertex set $[n]$ such that, $\{1,2\} \in E(G)$ and
\begin{equation} \label{Local Property 1.1}
\big\{ \{1,i\}, \{2,i\} \big\} \nsubseteq E(G), \qquad \text{for all } i>2.
\end{equation}
Let $I=I(\bar{G}) \subset S=K[x_1, \ldots, x_n]$ be the circuit ideal of $G$. Then,
\begin{itemize}
\item[\rm (i)] $\depth \frac{S}{I^\vee+(x_1, x_2)} \geq \depth \frac{S}{I^\vee} - 1$.
\item[\rm (ii)] $\depth \frac{S}{I^\vee \cap (x_1, x_2)} \geq \depth \frac{S}{I^\vee}$.
\end{itemize}
\end{lem}

\begin{proof}
Let $t:= \depth {S}/{I^\vee} \leq \dim {S}/{I^\vee} = n-2$.\\
(i) One can easily check that, condition (\ref{Local Property 1.1}) is equivalent to say that:
\begin{center}
for all $r>2$, there exists $F \in E(\bar{G})$ such that, $P_F \subset (x_1, x_2, x_r)$.
\end{center}
Therefore,
\begin{equation*}
\begin{split}
I^\vee =  \bigcap\limits_{F \in E(\bar{G})} P_F & = \left( \bigcap\limits_{F \in E(\bar{G})} P_F \right) \cap \left( (x_1, x_2, x_3) \cap \cdots \cap (x_1, x_2, x_n) \right) \\
& = \left( \bigcap\limits_{F \in E(\bar{G})} P_F \right) \cap (x_1, x_2, x_3 \cdots x_n)\\
& = I^\vee \cap (x_1, x_2, x_3 \cdots x_n).
\end{split}
\end{equation*}
Clearly, $x_3 \cdots x_n \in I^\vee$. Thus, from Mayer–Vietoris long exact sequence,
\begin{equation*}
\begin{split}
\cdots \rightarrow H_\mathfrak{m}^{i-1} \left( \tfrac{S}{I^\vee} \right) \oplus H_\mathfrak{m}^{i-1} \left( \tfrac{S}{(x_1, x_2, x_3 \cdots x_n)} \right) \rightarrow H_\mathfrak{m}^{i-1} \left( \tfrac{S}{I^\vee + (x_1, x_2)} \right)
\rightarrow  H_\mathfrak{m}^{i} \left( \tfrac{S}{I^\vee} \right) \rightarrow \cdots.
\end{split}
\end{equation*}
we have:
\begin{equation} \label{Local EQ 1.2}
H_\mathfrak{m}^{i-1} \left( \frac{S}{I^\vee + (x_1, x_2)} \right)=0, \qquad \text{for all } i<t \leq n-2.
\end{equation}
This proves (i).\\
(ii) From Mayer–Vietoris long exact sequence
\begin{equation*}
\begin{split}
\cdots \rightarrow H_\mathfrak{m}^{i-1} \left( \tfrac{S}{I^\vee + (x_1, x_2)} \right) \rightarrow  H_\mathfrak{m}^{i} \left( \tfrac{S}{I^\vee \cap (x_1, x_2)} \right) \rightarrow H_\mathfrak{m}^{i} \left( \tfrac{S}{I^\vee} \right) \oplus H_\mathfrak{m}^{i} \left( \tfrac{S}{(x_1, x_2)} \right) \rightarrow \cdots.
\end{split}
\end{equation*}
and (\ref{Local EQ 1.2}), we have:
\begin{equation*}
H_\mathfrak{m}^{i} \left( \frac{S}{I^\vee \cap (x_1, x_2)} \right)=0, \qquad \text{for all } i < t \leq n-2,
\end{equation*}
which completes the proof of (ii).
\end{proof}

\begin{thm} \label{Theorem edge subdivion}
Let $G$ be a graph on vertex set $[n]$ such that, $\{1,2\} \in E(G)$ and $\big\{ \{1,i\}, \{2,i\} \big\} \nsubseteq E(G)$, for all $i>2$. Let,
$$G_1 = \left( G \setminus \{1,2\} \right) \cup \big\{ \{0,1\}, \{0,2\} \big\}$$
be a graph on $\{0\} \cup [n]$ and $I=I(\bar{G}), J=I(\bar{G}_1)$ be circuit ideals in $S=K[x_0, x_1, \ldots, x_n]$. Then,
$$\reg (I) = \reg (J).$$
\end{thm}

\begin{proof}
By Remark~\ref{Passing to regularity}, it is enough to show that, $\depth {S}/{I^\vee} = \depth {S}/{J^\vee}$. Let $G_1= G \setminus \{1,2\}$ and $I_1=I(\bar{G}_1)$. Clearly, $I_1^\vee = (x_1, x_2) \cap I^\vee$ and
\begin{equation*}
\begin{split}
J^\vee & =  \left( \bigcap\limits_{i=3}^{n} (x_0, x_i) \right) \cap I_1^\vee \\
& = (x_0, x_3 \cdots x_n) \cap I_1^\vee.
\end{split}
\end{equation*}
Moreover, our assumption implies that for all $i>2$, there exists $F \in E(\bar{G})$ such that, $P_F \subset (x_1, x_2,  x_i)$. Therefore,
\begin{align} \label{Local EQ 1.3}
I_1^\vee + (x_0, x_3 \cdots x_n) & = (x_0, x_3 \cdots x_n , I_1^\vee) \nonumber \\
& = (x_0) + \left( x_3 \cdots x_n , \left[ (x_1, x_2) \cap \left( \bigcap\limits_{F \in \bar{G}} P_F \right) \right] \right) \nonumber \\
& = (x_0) + \left( (x_1, x_2, x_3) \cap \cdots \cap (x_1, x_2, x_n) \cap \left( \bigcap\limits_{F \in \bar{G}} P_F \right) \right) \nonumber \\
& = (x_0) + \left( \bigcap\limits_{F \in \bar{G}} P_F \right) = (x_0, I^\vee).
\end{align}
Now, consider Mayer-Vietoris long exact sequence
\begin{align}\label{Local exact sequence 1.1}
\cdots & \rightarrow H_\mathfrak{m}^{i-1} \left( \frac{S}{I_1^\vee} \right) \oplus  H_\mathfrak{m}^{i-1} \left( \frac{S}{(x_0, x_3 \cdots x_n)} \right) \rightarrow H_\mathfrak{m}^{i-1} \left( \frac{S}{(x_0, I^\vee)} \right) \rightarrow \nonumber \\
& \rightarrow H_\mathfrak{m}^{i} \left( \frac{S}{J^\vee} \right)  \rightarrow  H_\mathfrak{m}^{i} \left( \frac{S}{I_1^\vee} \right) \oplus  H_\mathfrak{m}^{i} \left( \frac{S}{(x_0, x_3 \cdots x_n)} \right) \rightarrow \cdots.
\end{align}
Let $t:= \depth \frac{S}{I^\vee} \leq \dim \frac{S}{I^\vee} = (n+1)-2$. Consider two cases:

\bigskip
\noindent\textbf{Case 1.} $t = (n+1)-2$.\\
In this case, using Lemma~\ref{lemma for edge subdivision}(ii), we have:
\begin{equation*}
(n+1) - 2 = \dim \frac{S}{I^\vee} = \depth \frac{S}{I^\vee} \leq \depth \frac{S}{I_1^\vee} \leq \dim \frac{S}{I_1^\vee} = (n+1) -2.
\end{equation*}
This means that, $\depth {S}/{I_1^\vee} = (n+1) -2$. Hence, by (\ref{Local exact sequence 1.1}), we have:
\begin{equation*}
H_\mathfrak{m}^{i-1} \left( \frac{S}{(x_0, I^\vee)} \right) \cong H_\mathfrak{m}^{i} \left( \frac{S}{J^\vee} \right), \qquad \text{for all } i<(n+1) -2
\end{equation*}
which implies that $\depth \frac{S}{J^\vee} = (n+1)-2 =t$.

\bigskip
\noindent\textbf{Case 2.} $t < (n+1)-2$.\\
Since $\depth S/I_1^\vee \geq t$, by Lemma~\ref{lemma for edge subdivision}(ii) and the exact sequence (\ref{Local exact sequence 1.1}), $H_\mathfrak{m}^{i} \left( \frac{S}{J^\vee} \right) = 0$, for all $i<t$ and we get the exact sequence
\begin{equation*}
0 \longrightarrow H_\mathfrak{m}^{t-1} \left( \frac{S}{(x_0, I^\vee)} \right) \longrightarrow H_\mathfrak{m}^{t} \left( \frac{S}{J^\vee} \right).
\end{equation*}
This implies that, $H_\mathfrak{m}^{t} \left( \frac{S}{J^\vee} \right) \neq 0$. Therefore, $\depth {S}/{J^\vee} = t$.
\end{proof}

Let $G$ be a graph without any cycle of length $3$ and $G_1$ a subdivision of $G$, that is, $G_1$ is obtained by adding some vertices on edges of $G$; then Theorem~\ref{Theorem edge subdivion} implies that $\reg \left( I(\bar{G}) \right) = \reg \left( I(\bar{G}_1) \right)$. As an application of the last reduction process, we state the following.

\begin{cor} \label{Regularity of cycles}
Let $C$ be a cycle of length $n>3$ and $I=I(\bar{C}) \subset S=K[x_1, \ldots, x_n]$ be the circuit ideal of ${C}$. Then,
\begin{itemize}
\item[\rm (i)] $\reg (I) =3$; in particular $I$ does not have linear resolution.
\item[\rm (ii)] If $G$ is not chordal graph, then the ideal $I(\bar{G})$ does not have linear resolution.
\end{itemize}
\end{cor}

\begin{proof}
(i) Let $E(C) = \big\{ \{1, 2\}, \{2, 3\}, \ldots, \{n-1, n\}, \{n,1\} \big\}$. We use induction on $n$. For $n=4$ an easy computation shows that, the minimal free resolution of $I(\bar{C})$ is:
$$
0\to S(-4) \to S^2(-2) \to I,
$$
which is not linear. Assume that $n>4$ and the theorem holds for cycles of length $n-1$. For a cycle $C$ of length $n$, let $C'$ be the graph $\left( C \setminus 1 \right) \cup \{1, 3\}$. Then $C'$ is a cycle of length $n-1$ and by induction hypothesis, $\reg I(\bar{C'}) = 3$. Using Theorem~\ref{Theorem edge subdivion}, we have $\reg I(\bar{C}) = \reg I(\bar{C'}) = 3$.

(ii) If $G$ is not chordal, then $G$ contains an induced cycle $C_n$ with $n >3$. Now, from (i) and Corollary~\ref{Induced subgraph} we conclude that the ideal $I(\bar{G})$ does not have linear resolution.
\end{proof}

Now, we state another reduction which is removing a simplicial vertex in a graph.

\begin{thm}\label{Removing simplicial vertex}
Let $G$ be a graph on $[n]$ and $v$ be a simplicial vertex of $G$. Let $G_1 = G \setminus v$ and $I= I(\bar{G}), J=I(\bar{G}_1)$ be the corresponding non-zero circuit ideals in $S = K[x_1, \ldots, x_n]$. Then,
$$\reg (I) = \reg (J).$$
\end{thm}

\begin{proof}
By Remark~\ref{Passing to regularity}, it is enough to show that, $\depth {S}/{I^\vee} = \depth {S}/{J^\vee}$. Without loss of generality, we may assume that, $N(v)= \{1, \ldots, v-1\}$ and $J \subset K[x_1, \ldots, \hat{x}_v, \ldots, x_n]$. Therefore, we have:
\begin{equation*}
I=J+(x_vx_i \colon \quad v<i\leq n).
\end{equation*}
Moreover, since $v$ is a simplicial vertex, we conclude that, $x_{v+1} \cdots x_n \in J^\vee$. Hence we have:
\begin{equation*}
\begin{split}
I^\vee 	& = J^\vee \cap \left( \bigcap\limits_{i=v+1}^{n}(x_v, x_i)  \right)\\
		& = J^\vee \cap (x_v, x_{v+1} \cdots x_n) \\
		& = \left( (x_v) \cap J^\vee \right) + (x_{v+1} \cdots x_n).
\end{split}
\end{equation*}
Clearly, $\left( (x_v) \cap J^\vee \right) \cap (x_{v+1} \cdots x_n) = (x_{v} \cdots x_n)$. Hence by Lemma~\ref{Marcel},
\begin{equation*}
H_\mathfrak{m}^i \left( \frac{S}{I^\vee} \right) \cong H_\mathfrak{m}^i \left( \frac{S}{(x_v) \cap J^\vee} \right), \qquad \text{for all } i<n-2.
\end{equation*}
Since $\dim S/I^\vee = n-2$, the above isomorphism and Lemma~\ref{Depth S/x_nI = Depth S/I} implies that, $\depth S/I^\vee = \depth S/J^\vee$.
\end{proof}

\noindent{\textbf{Remark.}} Let $G$ be a non-complete graph, $v$ be a simplicial vertex of $G$ and $G_1 = G \setminus v$. If $G_1$ is a complete graph, then the ideal $I = I(\bar{G}) = (x_vx_i \colon \quad \{v,i\} \in E(\bar{G}))$ is a non-zero ideal and
$$I^\vee = (x_v, \prod\limits_{\{v, i\} \in E(\bar{G})} x_i).$$
In particular, $I^\vee$ is Cohen-Macaulay and the ideal $I$ has a $2$-linear resolution (Theorem~\ref{Eagon-Reiner}).

If $G_1$ is not a complete graph, then Theorem~\ref{Removing simplicial vertex} implies that $\reg I(\bar{G}) =  \reg I(\bar{G}_1)$.

\bigskip
The following nice characterization of chordal graphs and Theorem~\ref{Removing simplicial vertex}, enable us to prove that the ideal $I(\bar{G})$ has a linear resolution, whenever $G$ is a chordal graph.
\begin{thm}[{\cite{Boland}, essentially \cite{Dirac}}] \label{Chordal Graphs have simplicial vertex}
A graph $G$ is chordal if and only if every induced subgraph of $G$ has a simplicial vertex.
\end{thm}

\begin{cor} \label{Chordal graphs have linear resolution 1}
If $G$ is a non-complete chordal graph, then the ideal $I = I(\bar{G})$ has a $2$-linear resolution over any filed $K$.
\end{cor}

\begin{proof}
Let $G$ be a non-complete chordal graph. By Theorem~\ref{Chordal Graphs have simplicial vertex}, $G$ has simplicial vertex $v$. If $G_1 = G \setminus v$, then $G_1$ is again chordal graph. Now, the induction and Theorem~\ref{Removing simplicial vertex} together with the remark after Theorem~\ref{Removing simplicial vertex}, yield the conclusion.
\end{proof}

By Corollaries \ref{Regularity of cycles}(ii) and \ref{Chordal graphs have linear resolution 1} we have the following result which was first proved by Fr\"oberg in \cite{Fr}.

\begin{cor} \label{Diff Froberg}
A graph $G$ is chordal if and only if $I(\bar{G})$ has a linear resolution.
\end{cor}

\noindent The class of chordal graphs are contained in the class of decomposable graphs (c.f. \cite[Lemma 9.2.1]{HHBook}). Using our reduction processes, we can find the regularity of decomposable graphs in terms of its indecomposable components.

\begin{defn}[Decomposable Graph] \rm
Let $G$ be a graph on vertex set $[n]$. We say that $G$ is \textit{decomposable}, if there exists proper subsets $P$ and $Q$ of $[n]$ with $P \cup Q = [n]$ such that,
\begin{itemize}
\item[\rm (a)] $\{i,j\} \in E(G)$, for all $i,j \in P\cap Q$, $i \neq j$.
\item[\rm (b)] $\{i, j\} \notin E(G)$, for all $i \in P \setminus Q$ and $j \in Q \setminus P$.
\end{itemize}
\end{defn}

\begin{rem}[Regularity of Decomposable Graphs] \rm
Let $G$ be a decomposable graph and $P,Q$ be proper subsets of $V(G)=[n]$ which satisfies in the mentioned conditions.
\begin{itemize}
\item If both of $G_P$ and $G_Q$ are complete graphs, then:
$$I(\bar{G}) = (x_iy_j \colon \quad i \in P \setminus Q, \quad j \in Q \setminus P).$$
Hence,
$$I(\bar{G})^\vee = \left( \prod\limits_{i \in P \setminus Q}x_i, \prod\limits_{i \in Q \setminus P}y_i \right)$$
which is Cohen-Macaulay of dimension $n-2$. Thus, $\reg I(\bar{G}) =2$, by Theorem~\ref{Eagon-Reiner}.

\item If $G_P$ is complete graph but $G_Q$ is not complete graph, then all $v \in P \setminus Q$ are simplicial vertex. Hence by Theorem~\ref{Removing simplicial vertex}, $\reg I(\bar{G}) = \reg I \left( \overline{G \setminus v} \right)$.

If $|P|=1$, we conclude that $\reg I(\bar{G}) = \reg I(\bar{G}_Q)$. Otherwise, the graph $G'=G \setminus v$ is again decomposable with the components $P'=P \setminus v$ and $Q$. Note that, $G'_{P'}$ is again a complete graph. Going on this argument, we conclude that, $\reg I(\bar{G}) = \reg I(\bar{G}_Q)$.

\item If non of $G_P$ and $G_Q$ are complete graphs, then Theorem~\ref{Theorem Regularity of Union of graphs} implies that,
$\reg I(\bar{G}) = \max \{ \reg I(\bar{G}_P), \; \reg I(\bar{G}_Q) \}$.
\end{itemize}
\end{rem}

\begin{rem} \rm
Let $G$ be a (indecomposable) graph. After our reduction processes (Theorems \ref{Theorem edge subdivion} and \ref{Removing simplicial vertex}), finally we get a graph $G'$ with $\reg I(\bar G) = \reg I(\bar{G}')$ and $G'$ has neither a simplicial vertex nor a subdivision. If at least one of the connected components of $G'$ has cycle of length greater that $3$, then $I(\bar G)$ does not have a $2$-linear resolution (Corollary~\ref{Regularity of cycles}(ii)).

But, sometimes we are not able to do more reduction on a graph. For example, if $G$ is the Peterson graph or the following Hamiltonian graph, then we cannot apply our reduction process to further simplify $G$.
\begin{center}
\psset{xunit=0.7cm,yunit=0.7cm,algebraic=true,dotstyle=o,dotsize=3pt 0,linewidth=0.8pt,arrowsize=3pt 2,arrowinset=0.25}
\begin{pspicture*}(-0.1,-1.9)(12.9,4.5)
\pspolygon[linestyle=none,fillstyle=solid](10.02,3.72)(7.66,2)(8.57,-0.78)(11.49,-0.77)(12.39,2.01)
\pspolygon[linestyle=none,fillstyle=solid](10.02,2.46)(8.96,1.68)(9.37,0.43)(10.69,0.44)(11.09,1.69)
\psline(2.34,3.72)(0,2)
\psline(0,2)(0.91,-0.76)
\psline(0.91,-0.76)(3.82,-0.74)
\psline(3.82,-0.74)(4.7,2.03)
\psline(4.7,2.03)(2.34,3.72)
\psline(2.34,3.72)(2.34,2.56)
\psline(3.4,1.76)(4.7,2.03)
\psline(3.02,0.34)(3.82,-0.74)
\psline(1.62,0.34)(0.91,-0.76)
\psline(1.2,1.76)(0,2)
\psline(2.34,2.56)(1.62,0.34)
\psline(2.34,2.56)(3.02,0.34)
\psline(1.62,0.34)(3.4,1.76)
\psline(3.4,1.76)(1.2,1.76)
\psline(1.2,1.76)(3.02,0.34)
\psline(10.02,3.72)(7.66,2)
\psline(7.66,2)(8.57,-0.78)
\psline(8.57,-0.78)(11.49,-0.77)
\psline(11.49,-0.77)(12.39,2.01)
\psline(12.39,2.01)(10.02,3.72)
\psline(10.02,2.46)(8.96,1.68)
\psline(8.96,1.68)(9.37,0.43)
\psline(9.37,0.43)(10.69,0.44)
\psline(10.69,0.44)(11.09,1.69)
\psline(11.09,1.69)(10.02,2.46)
\psline(10.02,3.72)(10.02,2.46)
\psline(7.66,2)(8.96,1.68)
\psline(9.37,0.43)(8.57,-0.78)
\psline(10.69,0.44)(11.49,-0.77)
\psline(11.09,1.69)(12.39,2.01)
\rput[tl](0.2,-1.3){Peterson Graph}
\rput[tl](7.6,-1.3){Hamiltonian Graph}
\psdots(2.34,3.72)
\psdots(0,2)
\psdots(0.91,-0.76)
\psdots(3.82,-0.74)
\psdots(4.7,2.03)
\psdots(2.34,2.56)
\psdots(1.2,1.76)
\psdots(3.4,1.76)
\psdots(3.02,0.34)
\psdots(1.62,0.34)
\psdots(10.02,3.72)
\psdots(7.66,2)
\psdots(8.57,-0.78)
\psdots(11.49,-0.77)
\psdots(12.39,2.01)
\psdots(10.02,2.46)
\psdots(8.96,1.68)
\psdots(9.37,0.43)
\psdots(10.69,0.44)
\psdots(11.09,1.69)
\end{pspicture*}
\end{center}
\end{rem}

\email{morales@ujf-grenoble.fr}
\email{yazdan@iasbs.ac.ir}
\email{rashidzn@iasbs.ac.ir}

\end{document}